\documentclass[12pt]{amsart}
\usepackage{amssymb}
\usepackage{latexsym}
\usepackage{graphicx}
\usepackage{subfigure}
\usepackage{tikz}
\usepackage[linktocpage=true]{hyperref}
\usepackage{ulem}

\usepackage{color}

\hypersetup{ colorlinks   = true, %Colours links instead of ugly boxes
urlcolor  = blue, %Colour for external hyperlinks
linkcolor    = blue, %Colour of internal links
citecolor   = blue %Colour of citations
}

\def\beq{\begin{equation}}
\def\eeq{\end{equation}}

\def\sp{\mathrm{sp}}
\def\Sp{\mathrm{Sp}}

\def\det{\mathrm{det}\ }

\def\Id{\mathrm{Id}}
\def\bm{\begin{matrix}}
\def\em{\end{matrix}}

\newcommand{\A}{{\mathbb A}}

\newcommand{\R}{{\mathbb R}}

\newcommand{\C}{{\mathbb C}}

\newcommand{\B}{{\mathbb B}}

\newcommand{\N}{{\mathbb N}}
\newcommand{\I}{{\mathbb I}}

\newcommand{\J}{{\mathbb J}}

\newcommand{\ri}{{\rm i}}
\newcommand{\CA}{{\mathcal A}}
\newcommand{\CC}{{\mathcal C}}

\newcommand{\CF}{{\mathcal F}}

\newcommand{\CB}{{\mathcal B}}
\newcommand{\CH}{{\mathcal H}}
\newcommand{\CI}{{\mathcal I}}
\newcommand{\CL}{{\mathcal L}}
\newcommand{\CM}{{\mathcal M}}

\newcommand{\CX}{{\mathcal X}}

\newcommand{\CS}{{\mathcal S}}

\newcommand{\CU}{{\mathcal U}}

\newcommand{\CQ}{{\mathcal Q}}

\newcommand{\CT}{{\mathcal T}}

\newtheorem{thm}{Theorem}[section]

\newtheorem{prop}[thm]{Proposition}
\newtheorem{lemma}[thm]{Lemma}

\theoremstyle{definition}
\newtheorem{remark}[thm]{Remark}

\newcommand{\la}{\langle}
\newcommand{\ra}{\rangle}

\definecolor{gainsboro}{rgb}{0.86, 0.86, 0.86}
\definecolor{gray-light}{rgb}{0.75, 0.75, 0.75}
\definecolor{gray}{rgb}{0.5, 0.5, 0.5}

\begin{document}

\title[]{}

\title[]{Generalized Reducibility and Growth of Sobolev Norms}

\author{Zhenguo Liang}
\address{School of Mathematical Sciences and Key Lab of Mathematics for Nonlinear Science, Fudan University, Shanghai 200433, China}
\email{zgliang@fudan.edu.cn}
\thanks{Z. Liang was partially supported by NSFC grant 12531006 and partially supported by the New Cornerstone Science Foundation through the New Cornerstone Investigator Program. }
	
\author{Zhiyan Zhao}
\address{Universit\'e C\^ote d'Azur, CNRS, Laboratoire J. A. Dieudonn\'{e}, 06108 Nice, France}
\email{zhiyan.zhao@univ-cotedazur.fr}
\thanks{Z. Zhao was partially supported by the French government through the National Research Agency (ANR) grant for the project KEN ANR-22-CE40-0016 and partially supported by NSFC grants 12271091, 12471178.}

%\keywords{}

%\date{\today}

\begin{abstract}
We introduce {\it generalized reducibility} to study the long-time behavior of solutions to time-dependent quadratic quantum Hamiltonians. As an application of this notion, for many prescribed sub-exponential growth rates $f(t)$, either monotone or oscillatory, we explicitly construct time-decaying perturbations of the one-dimensional quantum harmonic oscillator such that the Sobolev norms of solutions grow at the rate $f(t)$.
\end{abstract}

%\begin{keyword}
%%% keywords here, in the form: keyword \sep keyword
%1-d quantum harmonic oscillator; time quasi-periodic; reducibility; growth of Sobolev norms
%%% PACS codes here, in the form: \PACS code \sep code
%
%%% MSC codes here, in the form: \MSC code \sep code
%%% or \MSC[2008] code \sep code (2000 is the default)
%\MSC[2010] 	35Q40; 35Q41; 47G30
%\end{keyword}

%% \linenumbers

\maketitle

\section{Introduction}

This paper focuses on the Sobolev norms of solutions to the quadratic Hamiltonian PDE
\begin{equation}\label{orig-equ-1}
\frac{1}{{\rm i}}\partial_t \psi=H(t, Z)\psi,\quad t\geq t_0,
\end{equation}
with initial condition $\psi(t_0, \cdot)\in L^2(\R^n)$, where
$$Z^*=(D^*, X^*)=(D_1,\cdots, D_n,X_1,\cdots, X_n)$$ with $D_j=-\ri\partial_{x_j}$, and $X_j$ denotes multiplication by $j-$th coordinate function, i.e. $(X_j f)(x)=x_j f(x)$.  The linear operator $H(t,Z)$ in Eq. (\ref{orig-equ-1}) can be decomposed as
\begin{itemize}
\item {\it homogeneous part} $\CQ_{\CA(t)}(Z):=-\frac12\la Z,\CA(t)\J_nZ\ra$ with
\begin{eqnarray*}
\CA(\cdot)&=&\left(\begin{array}{cc}A_{11}(\cdot) & A_{12}(\cdot) \\ A_{21}(\cdot) & -A_{11}(\cdot)^* \end{array}\right)\in C_b^{0}([t_0,\infty[, \sp(n,\R)),\\
\J_n:&=&\left(\begin{array}{cc}
0 & \I_n \\ -\I_n & 0 \end{array}\right), 	
\end{eqnarray*}
\item {\it linear part} $\CL_{\ell(t)}(Z):=\la \ell(t),Z\ra$ with
$$\ell(\cdot)=(l_1(\cdot)^*,l_2(\cdot)^*)^* \in C_b^{0}([t_0,\infty[,\R^{2n}),$$
\item {\it scalar part} $c(t)$, with $c\in L^1_{\rm loc}([t_0,\infty),\R)$.
\end{itemize}
The quadratic quantum Hamiltonian $H(t, Z)$ is the Weyl quantization of the quadratic classical Hamiltonian
\begin{equation}\label{classicalHam}
h(t,\xi,x) =-\frac12\la z,\CA(t)\J_n z\ra+\la \ell(t), z\ra + c(t),\qquad \xi,x\in \R^{n},
	\end{equation}
	where $z\in \R^{2n}$ such that $z^*=(\xi^*,x^*)$, and the three terms of $h(t,\xi,x)$ can also be called {\it homogeneous part}, {\it linear part} and {\it scalar part} respectively in the classical Hamiltonian sense.
		
As a typical example of the quantum Hamiltonian $H(t,Z)$, the $n-$dimensional quantum harmonic oscillator (QHO for short),
\begin{equation}\label{nD-QHO}
\CT:=\CQ_{\J_n}(Z)=\frac12(\la D,D\ra+ \la X, X\ra),
\end{equation}
as well as its perturbations, is well investigated in many recent works and it is closely related to the present paper.	
	
\subsection{Generalized reducibility}

For a specific quantum Hamiltonian $H(t,Z)$ with constant coefficients, there exists a straightforward approach to analyzing the propagators of Eq. (\ref{orig-equ-1}). This includes, for  instance, Mehler's formula for the $n-$dimensional QHO $\CT$ defined in (\ref{nD-QHO}). The qualitative properties of the solutions can then be investigated through direct computations.

Consequently, to efficiently describe the behavior of solutions to the time-dependent equation, it is practical to remove the time-dependence from the original equation via an appropriate transformation. This process is referred to as {\it reducibility}.
A considerable amount of prior research has been devoted to the reducibility of
PDEs with time-dependent coefficients; see, for example, \cite{Bam2018, Bam2017, BGMR2018, BGMRV2025, BM2018} and the references therein.
%More recently, Bambusi-Gr\'ebert-Maspero-Robert-Villegas Blas \cite{BGMRV2025} have investigated the reducibility
%of a class of Landau Hamiltonians. Employing novel KAM algorithms tailored to
%the classical dynamics, they demonstrate that the Landau Hamiltonian is reducible to a twodimensional
%harmonic oscillator, thereby establishing boundedness of the dynamics.
Moreover, the study of the growth of Sobolev norms (defined in Section~\ref{sec_intro_growth}) of solutions is  closely connected to the reducibility of the underlying quantum Hamiltonian, particularly in the linear setting.

\medskip

In the works mentioned above, reducibility is typically achieved through an $L^{2}-$unitary transformation that is {\it bounded  in time on Sobolev spaces} (defined in Section \ref{sec_intro_growth}), and the growth of Sobolev norms is subsequently inferred by conjugating the equation to one with constant coefficients.
We propose a different perspective. Instead of requiring the conjugating transformation to be bounded on Sobolev spaces, we allow an $L^{2}-$unitary transformation that is not necessarily Sobolev-bounded. This transformation conjugates the original equation to a constant-coefficient equation whose solutions have uniformly bounded Sobolev norms, more precisely, to the constant-coefficient equation
 $$\frac{1}{{\rm i}}\partial_t \varphi=\CT \varphi \quad {\rm with}  \ \CT \ {\rm defined \  in} \ (\ref{nD-QHO}). $$
The growth of the Sobolev norms for solutions of the original equation is then recovered through a detailed analysis of the mapping properties of the conjugating transformation itself.
In contrast to reducibility in the classical sense, such an $L^{2}-$unitary conjugation can always be achieved. We refer to this notion as {\it generalized reducibility}.
To distinguish it from the previously studied notion of reducibility, we refer to reducibility in the previous sense as {\it normal reducibility}.

\begin{thm}\label{thm-reduc}{\bf (Generalized reducibility)}
The quadratic quantum Hamiltonian (\ref{orig-equ-1})
is {\it reducible in the generalized sense}, i.e., there exists an $L^2-$unitary transformation $$\psi(t)=\CU(t)\varphi(t),$$
such that the quantum Hamiltonian (\ref{orig-equ-1})
is conjugated to the constant-coefficient equation
$$\frac{1}{{\rm i}}\partial_t \varphi=\CT \varphi.$$
\end{thm}

A more detailed description in the Sobolev setting, together with a comparison between normal reducibility and generalized reducibility, will be presented in Section~\ref{sec_intro_growth}.

\subsection{Growth of Sobolev norms}\label{sec_intro_growth}

With the generalized reducibility obtained in Theorem \ref{thm-reduc}, we are able to describe the growth of {\it Sobolev norms} for quadratic quantum Hamiltonians, especially for the quantum Hamiltonians that are irreducible in the normal sense.

The {\it Sobolev space} (also referred to as {\it $\CH^s-$space} for a given $s\geq 0$) is defined as
\begin{equation}\label{defiSobo}
\CH^s(\R^n):=\left\{u\in L^2(\R^n):\CT^{\frac{s}{2}} u \in L^2(\R^n)\right\},
\end{equation}
with $\CT$ being the QHO defined as in (\ref{nD-QHO}), and equipped with the {\it Sobolev norm} (also known as {\it $\CH^s-$norm})
\begin{equation}\label{defiSoboNorm}
\|u\|_{s}:=\left(\|\CT^{\frac{s}{2}} u\|^2_{L^2}+\|u\|^2_{L^2}\right)^{\frac12}.
\end{equation}
Several equivalent forms of the $\CH^s-$norm will be given in Section \ref{sec_pre}.

\medskip

As an important issue in mathematical physics, the behavior of solutions to Hamiltonian PDEs in Sobolev spaces was first proposed by Bourgain \cite{Bou96} and subsequently developed through a series of pioneering works (see, e.g., \cite{CKSTT2010, GG2016, GK2015, HPTV2015}). This topic has attracted sustained interest over the past decades.

Beyond the boundedness of Sobolev norms and the existence of ``weak" upper bounds (such as bounds of order $t^\varepsilon$ or $\log(t)$) on their growth (see \cite{BGMR2018, BGMR2021, BLM20, BL2025} and the references therein), a more challenging problem is to establish the existence of unbounded trajectories in Sobolev spaces. This question is naturally related to weak turbulence phenomena and energy cascade mechanisms, as well as to the precise determination of the time growth rates of Sobolev norms. Over the past decades, substantial progress has been made in this direction for a wide class of Hamiltonian PDEs.

As mentioned in the previous subsection, it is effective to study the long-time behavior of solutions through normal reducibility of the original quantum Hamiltonian, especially for the linear Hamiltonians. 
\begin{itemize}
\item %Following the approach of \cite{GY00},
For a one-dimensional (1D) QHO with a time-periodic linear potential, Bambusi-Gr\'ebert-Maspero-Robert \cite{BGMR2018} proved polynomial growth of order $t^s$ for the $\CH^s-$norms by reducing the system to a transport equation.
\item For 1D QHO with time quasi-periodic perturbations, which are quadratic polynomials of $(D,X)$, various growth rates of $\CH^s-$norms have been observed depending on the normal forms of reducibility \cite{LZZ2021, LLZ2022}:  $t^{s}$ polynomial growth for parabolic normal form, exponential growth for hyperbolic normal form, and  $t^{2s}$ polynomial growth when reducible to the Stark Hamiltonian.

\item More generally, in the $n$-dimensional setting, a complete classification of Sobolev norm growth can be obtained under the assumption of reducibility (see Theorem 1.2 of \cite{LLZ2025}). In particular, if the quantum Hamiltonian (\ref{orig-equ-1}) is reducible in the sense of Theorem 1.2 of \cite{LLZ2025}, then the long-time growth behavior of $\CH^s$-norms can be fully characterized according to the associated normal form.

\item Beyond exact reducibility, the notion of {\it almost reducibility} in the time quasi-periodic setting also plays a significant role in the study of Sobolev norm growth. In this framework, one typically observes oscillatory growth phenomena, reflecting the presence of persistent but non-removable time dependence. Based on the almost reducibility scheme developed by Eliasson \cite{Eli1992} for time quasi-periodic linear systems, as long as the time-dependent error remains within a controllable range, the growth of Sobolev norms can still be effectively captured through the conjugating transformation. Meanwhile, the reduced system exhibits an almost stable behavior over long time scales. We refer to \cite{LZZ2024} for a detailed analysis.
\end{itemize}

%{\clo For Hamiltonian PDEs, almost reducibility is also useful to see the long-time behaviors of solutions.
%Eliasson \cite{E2017} showed almost reducibility for the quasi-periodic linear wave equation, through which a log-log-bound on Sobolev norms of solutions is obtained.
%Similar idea was employed by Bambusi and his collaborators \cite{BGMR2021, BLM20, BL2025} for time dependent Schr\"odinger equations, with a more and more regularized remaining part along with the iteration step, instead of the usual asymptotic smallness assumption. A $t^\epsilon-$upper bound, for arbitrary $\epsilon>0$, on Sobolev norms of solutions is obtain through such an argument.}

Another way to obtain unbounded trajectories is to construct perturbations that produce Sobolev norm growth.
For 1D QHO, Delort \cite{Del2014} (followed by a refined proof of Maspero \cite{Mas2018}) constructed time-periodic order-zero pseudo-differential operators as the perturbation such that some solutions exhibit $t^{\frac{s}{2}}-$polynomial growth of ${\CH}^s-$norms, and in an abstract setting, Maspero \cite{Mas2022, Mas2023} exploited further time-periodic perturbations and gave sufficient conditions such that some solution exhibits such polynomial growth, using the Mourre estimate. This strategy has also been applied by Maspero
and collaborators to the 1D fractional quasilinear Schr\"odinger equation \cite{MM2024} and to completely resonant QHO on $\R^2$ \cite{LMT2025}.

For the $2-$dimensional QHO, Faou-Rapha\"el \cite{FR23} constructed a time-decaying perturbation such that, after a suitable ``dilation transformation", the Sobolev norms of some solutions to the transformed equation remain bounded in time.
However, the corresponding solutions of the original system exhibit logarithmic growth in Sobolev norms. We remark that the underlying philosophy is similar to that of generalized reducibility, even though the setting and methods are different.

Based on the study in \cite{ST20} for linear Lowest Landau Level equations with a time dependent potential, Thomann \cite{Thomann20} constructed a perturbation given by the projection onto Bargmann-Fock space and obtained a traveling wave whose Sobolev norm grows polynomially in time.
In other settings, the logarithmic growth of Sobolev norms was also shown by Bourgain \cite{Bou99} for 1D and 2D linear Schr\"odinger equations with quasi-periodic potential, and by Haus-Maspero \cite{HM2020} for semiclassical anharmonic oscillators with regular time dependent potentials.

\subsubsection{Classical-quantum correspondence}

For the quadratic quantum Hamiltonian
\begin{equation}\label{orig-equ-2}
\frac{1}{{\rm i}}\partial_t \psi=(\CQ_{\CA(t)}(Z)+\CL_{\ell(t)}(Z))\psi,
\end{equation}
with $\CA(t)\in C_b^0(\R, \sp(n, \R))$ and $ \ell(\cdot)\in C_b^0(\R, \R^{2n}) $, the long-time behavior of its solutions is closely related to the classical affine system
\begin{equation}\label{orig-ODE-2}
\dot{z}=\CA(t)z+\ell(t).
\end{equation}
According to \cite[Theorem 1.4]{LLZ2025}, under the reducibility assumption on the quantum Hamiltonian (\ref{orig-equ-2}), the growth of $\CH^s-$norms for the solutions to Eq. (\ref{orig-equ-2}) is determined by the solution to the affine system (\ref{orig-ODE-2}). With generalized reducibility stated in Theorem \ref{thm-reduc}, such a classical-quantum correspondence
can be generalized by removing the reducibility assumption.

%$$\psi(t)=\CU(t)\varphi(t)=\CU(t)e^{{\rm i}t\CT }\CU(0)^{-1}\psi(0)$$

\begin{thm}\label{thmODEtoPDE}
For $s\geq 0$ and $t_0>0$, consider Eq. (\ref{orig-equ-2}) with initial condition $\psi(t_0)\in\CH^s\setminus\{0\}$. There is a constant $C>1$, depending on $\psi(t_0)$ and $z_{*}(t_0)$, such that
the corresponding solution $\psi(t)$ of Eq. (\ref{orig-equ-2}) satisfies
\begin{equation}\label{esti_Hs}
C^{-1}  \leq \frac{\|\psi(t)\|_s}{(\|z_*(t)\|+\|W(t)\|)^s}\leq C , \qquad  \forall \ t\geq t_0, \footnote{Throughout this paper, the norm of a real
$m\times m$ matrix is the Euclidean norm of its entries in
$\R^{m^2}$ (the Frobenius norm). In particular, for the fundamental solution matrix $W(t)$, this is
the Euclidean norm in $\R^{4n^2}$.}
\end{equation}
where $z_*(t)$ is a particular solution of the affine system (\ref{orig-ODE-2}), and
$W(t)$ is the fundamental solution matrix of the linear system $\dot{z}=\CA(t)z$ satisfying
$W(t_0)=\I_{2n}$.
In particular, if $\ell(\cdot)=0$ in (\ref{orig-ODE-2}), then
\begin{equation}\label{esti_Hs-homo}
C^{-1}  \leq \frac{\|\psi(t)\|_s}{\|W(t)\|^s}\leq C ,  \qquad  \forall \  t\geq t_0.
\end{equation}
\end{thm}

Since Eq. (\ref{orig-equ-2}) is not reducible in general, one cannot compute the growth of Sobolev norms by passing to a normal form. In particular, the lack of reducibility prevents us from conjugating the equation to a time-independent Hamiltonian for which the norm growth can be read off directly.
Generalized reducibility overcomes this difficulty by allowing conjugations that are unitary on $L^2$ but not necessarily bounded on $\CH^s$, thereby still making it possible to determine the growth of Sobolev norms.

The main idea in the proof of Theorem \ref{thmODEtoPDE} is precisely to exploit this generalized conjugation framework in order to transfer the growth mechanism from the transformed equation back to the original one.
The difference between normal reducibility and generalized reducibility can be summarized as follows.
$$
\begin{tabular}{|c|l|l|l|}
\hline
$\psi(t)=\CU(t)\varphi(t)$ & Normal reducibility & Generalized reducibility \\[2mm]
 \hline
$\CU(t)$ on $L^2$  & unitary & unitary \\[2mm]
\hline
$\CU(t)$ on $\CH^s$  & bounded  w.r.t. $t$ &  maybe unbounded w.r.t. $t$ \\[2mm]
\hline
$\|\varphi(t)\|_s$ & maybe unbounded  w.r.t. $t$ &  bounded w.r.t. $t$ \\[2mm]
\hline
$\|\psi(t)\|_s$  & determined by $\|\varphi(t)\|_s$ & determined by $\CU(t)$ on $\CH^s$ \\[2mm]
\hline
Realizability &  Not always realizable &  Realizable (Theorem \ref{thm-reduc}) \\[2mm]
\hline
\end{tabular}
$$

Moreover, generalized reducibility differs essentially from normal reducibility in the {\it diversity of possible growth behaviors} of Sobolev norms.
In \cite[Theorem 1.2]{LLZ2025}, we obtained a complete classification of the growth rates of Sobolev norms for Eq.~(\ref{orig-equ-2}) when the system is reducible in the normal sense.
In contrast, under generalized reducibility, solutions may exhibit a much richer variety of behaviors: one can realize many prescribed sub-exponential growth rates, either monotone or oscillatory, of Sobolev norms.
Theorems \ref{thmGR-mono} and \ref{thmGR-osci} give precise statements.

\subsubsection{Construction for prescribed growth rate}

According to Theorem~\ref{thmODEtoPDE}, since the growth of Sobolev norms for quadratic quantum Hamiltonians can be read off from the corresponding classical linear system, it is possible to construct quantum Hamiltonians whose solutions exhibit prescribed growth rates in Sobolev spaces.

\medskip

Given some $t_0>0$, let $\CM([t_0,\infty[)\subset C^2([t_0,\infty[, \R_+^*)$ be the class of $C^2$ functions on $[t_0,\infty[$ satisfying
\begin{eqnarray}
& &\inf_{t\geq t_0} f(t) > 0\label{Mclass1}\\
& &f(t)\to \infty \ {\rm non-decreasing}, \quad {\rm as} \ t\to\infty. \label{Mclass2}\\
& &\left(\frac{f'}{f}\right)(t) \to 0  \ {\rm non-increasing}, \quad {\rm as} \ t\to\infty. \label{Mclass3}
\end{eqnarray}
It is easy to verify that, for suitable $t_0$, $\CM([t_0,\infty[)$ includes a large class of sub-exponential growth rates such as
\begin{equation}\label{examp-M}
\begin{array}{lll}
e^{\log(t)^a} \ (a>0), & e^{{\rm li}(t)}, & e^{t^{\sigma}} \ (0<\sigma<1),\\[3mm]
 t, & \frac{t}{\log(t)}, & \underbrace{\log\log\cdots\log}_{k \ {\rm times}}(t) \ (k\in\N^*).
\end{array}\end{equation}
Moreover, if $f\in\CM([t_0,\infty[)$, then $f^\lambda\in\CM([t_0,\infty[)$ for any $\lambda>0$.

\begin{thm}\label{thmGR-mono}
For $f\in \CM([t_0,\infty[)$, there exists $\phi_f\in C^0([t_0,\infty[,\R)$
 such that the following holds for any given $s>0$.
 \begin{enumerate}
   \item\label{thmGR-mono-1} For the 1D homogeneous quadratic Hamiltonian PDE
\begin{equation}\label{QHe-abs}
\frac1\ri\partial_t u=\frac12\left(D^2 + (1+\phi_f(t))X^2\right) u,
\end{equation}
with $u(t_0)\in \CH^s\setminus\{0\}$,
there exists a constant $c_0>1$ such that the solution $u=u(t)$
satisfies
\begin{equation}\label{growth}
c_0^{-1}f(t)^{\frac{s}{2}}\leq \|u(t)\|_{s} \leq c_0f(t)^{\frac{s}{2}},\quad \forall \  t\geq t_0.
\end{equation}
   \item\label{thmGR-mono-2} If there exists $0<\kappa<2$ such that
   \begin{equation}\label{condisupp1}
0\leq t\left(\frac{f'}{f}\right)(t)\leq \kappa,\quad \forall \  t\geq t_0,
\end{equation}
then, for the 1D non-homogeneous quadratic Hamiltonian PDE
\begin{equation}\label{QHe-abs-nonhomo}
\frac1\ri\partial_t u=\frac12\left(D^2 + (1+\phi_f(t))X^2\right) u+a\sin(t) X u,\quad a\in\R^*,
\end{equation}
with $u(t_0)\in \CH^s\setminus\{0\}$,
there exists a constant $c_1>1$ such that
\begin{equation}\label{growth-lin}
c_1^{-1}t^s\leq \|u(t)\|_{s} \leq c_1t^s,\quad \forall \  t\geq t_0.
\end{equation}
\end{enumerate}
Moreover, if $f\in \CM([t_0,\infty[)$ satisfies $f''/f\to 0$ as $t\to\infty$, then
\begin{equation}\label{decaying-phi}
\phi_f(t)\to 0 \quad as \ \ \ t\to\infty.
\end{equation}
\end{thm}

\begin{remark}
It is easy to verify that, for suitable $t_0>0$, there exist many non-decreasing functions in $\CM([t_0,\infty[)$ that grow to $+\infty$ slower than $t^2$ and satisfy (\ref{condisupp1}). Typical examples include
$$f(t)=\mu t^\alpha \log(t)^\beta,\quad  \mu>0, \;\  0\leq \alpha<2, \;\  \beta\geq 1. $$
The assertion (2) in Theorem \ref{thmGR-mono} shows that a large class of monotone $o(t^s)-$growth behaviors of the $\CH^s-$norm obtained for the homogeneous equation (\ref{QHe-abs}) can be improved to $t^s-$growth by means of an order $1$ perturbation.
It was also shown by Bambusi-Gr\'ebert-Maspero-Robert \cite[Corollary A.2]{BGMR2018} that boundedness of the $\CH^s-$norm for the homogeneous equation~(\ref{QHe-abs}) can be upgraded to $t^s-$growth via such an order $1$ perturbation.
\end{remark}

Besides the non-decreasing growth, it is also possible to construct oscillatory growth of Sobolev norms. As an example, we have

\begin{thm}\label{thmGR-osci}
For the function
\begin{equation}\label{oscillatory_f}
f(t)= t^\frac13\left(1+\log(t) \sin^2\left(\sqrt{t}\right)\right),\quad t\geq e,
\end{equation}
there exists $\phi_f\in C^{\infty}([e,\infty[,\R)$ with
\begin{equation}\label{decay-oscillatory_f}
\phi_f(t)\to 0 \quad as \ \ \ t\to\infty,
\end{equation}
 such that, for any $s>0$, any initial condition $u(t_0)\in \CH^s\setminus\{0\}$ with $t_0\geq e$, there exists a constant $c_0>1$ such that (\ref{growth}) holds for the solution $u(t)$ to Eq. (\ref{QHe-abs}).
In particular,
$$\liminf_{t\to\infty}\frac{\|u(t)\|_{\CH^s}}{t^\frac{s}{6}}, \quad \limsup_{t\to\infty}\frac{\|u(t)\|_{\CH^s}}{t^\frac{s}{6}\log(t)^\frac{s}{2}}\in [c_0^{-1},   c_0]. $$
\end{thm}

Theorems \ref{thmGR-mono} and \ref{thmGR-osci} will be proved in Sections \ref{sec-GRmono} and \ref{sec-GRosci}, respectively, by verifying the hypotheses of an abstract result (Theorem \ref{thmGr}). The main idea is to transfer the computation to the associated affine system, made possible by generalized reducibility.

\subsection*{Acknowledgements}
Z. Zhao acknowledges support from the Key Laboratory Senior Visiting Scholarship at Fudan University during his stay in 2025 and 2026. 
He also thanks the School of Mathematics at Nanjing University for its hospitality during his visit in 2025 and acknowledges support from the NJU International Fellowship Initiative and the NJU International Research Seed Fund.

\section{Preliminaries}\label{sec_pre}
In this section, we recall the Schr\"odinger and Metaplectic representations and their fundamental properties.

%\subsection{Sobolev space, Sobolev norm and Fourier transform}\label{sec_Sobolev}
%With the $n-$D QHO in (\ref{nD-QHO}), the {\it Sobolev space} (also referred to as the {\it $\CH^s$-space} for a given $s \geq 0$) is defined as
%\begin{equation*}\label{Sobo_space-n}
%\CH^s(\R^n):=\left\{f\in L^2(\R^n):\CT^{\frac{s}{2}} f \in L^2(\R^n)\right\},\quad s\geq 0,
%\end{equation*}
%equipped with the graph norm, known as the {\it Sobolev norm} or {\it $\CH^s$-norm}:
%\begin{equation*}\label{Sobo_norm-n}
%\|f\|_{\CH^s(\R^n)}:=\left(\|\CT^{\frac{s}{2}} f\|^2_{L^2(\R^n)}+\|f\|^2_{L^2(\R^n)}\right)^{\frac12}.
%\end{equation*}
%{\clr and for $s<0$, $\CH^{s}(\R^n)=(\CH^{-s}(\R^n))'$ by duality. For convenience, we sometimes denote $L^2$ and $\CH^s$ instead of $L^2(\R^n)$ and $\CH^s(\R^n)$ respectively,
%when the dimension of the base space is explicitly stated, and we abbreviate $\|\cdot\|_{\CH^s}$ as $\|\cdot\|_{s}$.}

Define the {\it Fourier transform} ${\CF}$  on $L^2(\R^n)$ and its inverse ${\CF}^{-1}$ as follows:
\begin{eqnarray*}({\CF}f)(\xi) \, = \, \hat{f}(\xi)&=& \int_{\mathbb R^n} e^{- {\rm i}\la x, \xi\ra}  f(x) \, dx, \qquad f\in L^2(\R^n), \\
({\CF}^{-1}\hat{f})(x) \, = \, f(x)&=& \frac{1}{(2\pi)^n}\int_{\R^n}e^{{\rm i }\la x, \xi\ra} \hat{f}(\xi) \,  d\xi, \qquad \hat{f}\in L^2(\R^n).
\end{eqnarray*}
According to \cite{FR23}, for $u \in \CH^s$, the equivalent forms of the $\CH^s-$norm (\ref{defiSoboNorm}) can be expressed as follows:
\begin{eqnarray*}\label{Sobo-equi1}
\|u\|_{s} &\simeq &\| \langle D\rangle^s u \|_{L^2} + \|\langle X\rangle^s u \|_{L^2} \ \simeq  \ \| \langle X\rangle^s \hat{u} \|_{L^2} + \|\langle X\rangle^s u \|_{L^2} \\
&\simeq& \|u \|_{L^2}+ \||X|^s \hat{u} \|_{L^2} + \||X|^s u \|_{L^2},
\end{eqnarray*}
where the implicit constants in ``$\simeq$" depend only on the dimension $n$ and the regularity index $s$.

\subsection{Schr\"odinger representation}\label{Section_schro-repre}

Let ${\bf H_n}$ be the real $(2n+1)-$dimensional {\it Heisenberg group}, which is $\R^{2n+1}$ equipped with the group law
$$(p, q, t) (p', q', t') =\left(p+p', q+q', t+t'+\frac12(\la p,q' \ra- \la q, p'\ra)\right),$$
where $p, q,p',q' \in \R^{n}$ and  $t,t'\in \R$. The map $\rho$, from ${\bf H_n}$ to the group of unitary operators on $L^2(\R^n)$, is defined as
$$\rho(p,q, t) =e^{{\rm i}(\la p,D\ra+\la q, X\ra+t\Id)} = e^{ {\rm i }t} e^{{\rm i}(\la p,D\ra+\la q, X\ra)}, \quad (p,q, t)\in {\bf H_n},$$
that is, for $u\in L^2(\R^n)$,
\begin{equation}\label{rhode}
\left(\rho(p,q,t) u\right)(x)=e^{{\rm i}t+ {\rm i}\la q, x\ra + \frac{\rm i}{2} \la p,q\ra} u(x+p),
\end{equation}
$\rho$ then becomes an irreducible unitary representation of ${\bf H_n}$ on the Hilbert space $L^2(\R^n)$, known as the {\it Schr\"odinger representation}. Since the variable $t$ always acts in a straightforward manner, it is often convenient to omit it entirely.

The {\it Schr\"odinger representation} $\rho$ is defined from $\R^{2n}$ to $\CB(L^2)$. For $p,q\in\R^n$, it can be explicitly formulated as
\begin{equation}\label{SchroRep-L2}
\left(\rho(v)u\right)(x) = e^{ {\rm i}\la q, x\ra +\frac{\rm i}{2} \la p,q\ra} u(x+p),\quad v=\begin{pmatrix} p \\ q \end{pmatrix}, \quad u\in L^2(\R^n).\end{equation}
It is easy to verify that $\rho(v)$ is $L^2-$unitary and belongs to $\CB(\CH^s)$ for any $s\geq 0$. Moreover, for $p_1,q_1,p_2,q_2 \in \R^n$, a straightforward computation shows that
\begin{equation}\label{schrepmult}
\rho\begin{pmatrix} p_1 \\ q_1 \end{pmatrix}\rho\begin{pmatrix} p_2 \\ q_2\end{pmatrix}=e^{\frac{\rm i}{2}(\la p_1,q_2\ra -\la q_1,p_2\ra)}\rho\begin{pmatrix} p_1+p_2 \\ q_1+q_2
\end{pmatrix},
\end{equation}
which implies
$\rho(v)^{-1}=\rho(-v)$ for any $v\in \R^{2n}$.

\subsection{Metaplectic representation}\label{Section_meta-repre}

The {\it symplectic group}, denoted by $\Sp(n,\R)$, comprises $2n\times 2n$ real matrices that preserve the symplectic form. Specifically, a matrix $\A= \begin{pmatrix} A & B\\ C& F \end{pmatrix}$ belongs to $\Sp(n,\R)$, with $A$, $B$, $C$, and $F$ being its four $n\times n$ blocks, if and only if the following conditions are satisfied:
$$A^{*} C= C^{*} A,\quad  B^{*}F= F^{*} B,\quad A^* F- C^* B =\I_n.$$
Such a matrix $\A\in \Sp(n,\R)$ is referred to as a {\it symplectic matrix}.

The {\it symplectic Lie algebra}, $\sp(n, \R)$, consists of all matrices $\CA\in {\rm gl}(2n,\R)$ such that $e^{t\CA} \in \Sp(n,\R)$ for every $t\in \R$. In other words, $\CA= \begin{pmatrix} A & B\\ C& F \end{pmatrix} \in \sp(n, \R)$, with $A$, $B$, $C$, and $F$ as the four $n\times n$ blocks of $\CA$, if and only if the following hold:
$$F=-A^*, \quad B=B^{*}, \quad C=C^*. $$
A matrix $\CA\in \sp(n,\R)$ is known as a {\it Hamiltonian matrix}.
		
%Let ${\bf Q}$ represent the space of real homogeneous quadratic polynomials on $\R^{2n}$, equipped with the {\it Poisson bracket}. Specifically, for smooth functions $Q_1$ and $Q_2$ on $\R^{2n}$, the Poisson bracket is defined as
%$$\left\{Q_1(\xi, x), Q_2(\xi, x)\right\}:= \sum\limits_{j=1}^n \left( \frac{\partial Q_1}{\partial \xi_j} \frac{\partial Q_2}{\partial x_j}-\frac{\partial Q_1}{\partial x_j} \frac{\partial Q_2}{\partial \xi_j}\right). $$
%\begin{prop}(Proposition 4.42 in \cite{Fol1989}) The mapping $\CA\mapsto \CQ_{\CA}(\xi, x)$ constitutes a Lie algebra isomorphism from $sp(n,\R)$ to ${\bf Q}$.
%\end{prop}
%
%\begin{cor}
%If $\sigma$ or $\tau$ is a polynomial of degree smaller than 2, then
%$$[\sigma(D, X), \tau(D, X)]=\frac{1}{\rm i}\{\sigma, \tau\}(D,X), $$
%where we define the Weyl quantization to be operator $\sigma(D, X)$ acting on $f\in L^2(\R^n)$ by the formula
%$$\sigma(D, X)f =\frac{1}{(2\pi)^n}\int_{\R^n}\int_{\R^n} e^{{\rm i}(x-y, \xi)}\sigma\left(\xi, \frac{x+y}{2}\right)f(y)dyd\xi. $$
%\end{cor}
								
The {\it Metaplectic representation} $\CM$ of $\Sp(n,\R)$ is defined as follows. 
Each $\A\in \Sp(n,\R)$ induces an automorphism $T_{\A}$ of the Heisenberg group ${\bf H_n}$ by
$$T_{\A}(p,q, t)=\left(\A\begin{pmatrix}p  \\ q \end{pmatrix}, t\right),\qquad p,q\in \R^n, \qquad t\in\R. $$
Since the Schr\"odinger representation $\rho$, defined in (\ref{rhode}), is an irreducible unitary representation of ${\bf H_n}$ on $L^2(\R^n)$, the composite $\rho\circ T_{\A}$ is also an irreducible unitary representation of ${\bf H_n}$, satisfying
$$(\rho\circ T_{\A})(0,0, t)=e^{{\rm i}t},\qquad t\in\R. $$
By the Stone-von Neumann theorem, the two irreducible unitary representations $\rho$ and $\rho\circ T_{\A}$ are equivalent, meaning there exists a unitary operator $\CM(\A)$ on $L^2(\R^n)$ such that
\begin{equation}\label{hengdeng1}
(\rho\circ T_{\A})(\CX) =\CM(\A)  \rho (\CX)   \CM(\A)^{-1}, \qquad \CX\in {\bf H_n}.
\end{equation}
As demonstrated in \cite{Fol1989}, $\CM(\A)$ can be uniquely chosen, up to factors $\pm 1$, so that $\CM$ becomes a double-valued unitary representation of $\Sp(n,\R)$. Hence,
\begin{equation}\label{product-Meta}
\CM(\A \mathbb B )=\pm \,  \CM(\A) \CM(\B),\qquad \A, \B\in\Sp(n,\R),
\end{equation}
and, specifically, $\CM(\A) \CM\left(\A^{-1}\right)=\pm \, \Id$.
Therefore, $\CM$ maps $\Sp(n,\R)$ into the group of unitary operators on $L^2(\R^n)$ modulo ${\pm \, \Id}$ and is referred to as the {\it Metaplectic representation} of $\Sp(n,\R)$. In explicit formulations, the ambiguity of $\pm 1$ typically manifests as the ambiguity in the sign of a square root. For simplicity, we often omit the sign $\pm1$.

%Given $\A\in \Sp(n,\R)$, the operator $\CM(\A)$ is well-defined on $L^2(\R^n)$, and consequently, on $\CH^s(\R^n)$ for $s\geq 0$, as specified in (\ref{hengdeng1}). For $s > 0$, we extend the definition of $\CM(\A)$ to $\CH^{-s}(\R^n)$ using duality:
%\begin{equation}\label{dfHms}
%\la \CM(\A)u, v\ra=\la u, \CM(\A^{-1}) v\ra, \qquad u \in \CH^{-s}(\R^n),  \quad  v\in \CH^s(\R^n).
%\end{equation}
%								
%Define the infinitesimal version $d\CM$ of the Metaplectic representation as
%\begin{equation}\label{defi-infinitesimal}
%d\CM(\CA):=\left.\frac{d}{dt} \CM\left(e^{t\CA}\right)\right|_{t=0}, \qquad \CA\in sp(n, \R).
%\end{equation}
%In this context and subsequently, the sign of $\CM(e^{t\CA})$ is chosen to ensure continuity in $t$ and equality to $\mathbb Id$ at $t=0$. Under this framework, the following theorem is established:
%\begin{thm}(Theorem 4.45 in \cite{Fol1989})\label{Thm4.45}
%For any $u$ in the Schwartz space $\CS(\R^n)$ and $\CA\in sp(n, \R)$, $$ d\CM(\CA)u= {\rm i}\CQ_{\CA}(Z) u.$$
%\end{thm}
%There exists a remarkable correspondence between the solution of ODE and  the quantum PDE as the following.
%It is well known that
%$z(t)=e^{\CA t}z_0$ satisfies
%the Hamiltonian equation
%\begin{equation}\label{ODE-pre}
%z'(t) =\CA z(t), \qquad z(0)=z_0\in \R^{2n}.
%\end{equation}
% Denote $\psi(t)=\CM(e^{t\CA})u$.
%\begin{thm}(\cite{Fol1989})
%$\psi(t)$ satisfies the quantum Hamiltonian PDE
%\beq\label{PDE-pre}
%\frac{1}{{\rm i}}\partial_t \psi= \CQ_{\CA}(Z) \psi(t,x) , \qquad \psi(0)=u\in \CH^{s}(\R^n),\eeq
%where the matrix $\CA\in sp(n, \R)$ and $s\in \R$.
%\end{thm}

\medskip										
										
For specific matrices $\A\in \Sp(n, \R)$, explicit formulas of the Metaplectic representation can be provided:
\begin{thm}\cite{Fol1989}\label{ThmFormMeta}
Let $\A= \begin{pmatrix}A & B\\C & F\end{pmatrix}\in \Sp(n, \R)$ and $u\in L^2(\R^n)$.
\begin{itemize}
\item [(i)] If $\det A\neq 0$, then $$\displaystyle (\CM(\A) u)(x) =\frac{1}{(2\pi)^n}(\det A)^{-\frac12} \int_{\R^n} e^{{\rm i} S(\xi,x)} \hat{u}(\xi) \, d\xi,$$
where $$S(\xi,x):= -\frac12\la x, CA^{-1} x\ra+\la \xi, A^{-1}x\ra+\frac12\la\xi, A^{-1}B \xi\ra . $$
\item [(ii)] If $\det B\neq 0$, then
$$\displaystyle (\CM(\A) u)(x) =\left(\frac{\rm i}{2\pi}\right)^{\frac{n}{2}}(\det B)^{-\frac12} \int_{\R^n} e^{ {\rm i} T(x,y)} u(y) \, dy,$$
where $$T(x,y)= -\frac12\la x, FB^{-1} x\ra+\la y, B^{-1}x\ra-\frac12\la y, B^{-1}A  y\ra .$$
\end{itemize}
\end{thm}

%In the new scale, we still have
\begin{prop}(\cite[Proposition 3.2]{LLZ2025})\label{boundofmeta}
For $\A\in \Sp(n,\R)$, $s\geq 0$, we have
 \begin{equation}\label{equiMeta}
\|\A\|^{-s}\|u\|_{s} \lesssim \|\CM(\A)u\|_{s}\lesssim \|\A\|^{s} \|u\|_{s},\qquad \forall \  u\in\CH^s.
\end{equation}
Furthermore, for $s>0$ and $u\in\CH^s\setminus \{0\}$, it holds that
\begin{equation}\label{optimallower-Meta}
\|\CM(\A)u\|_{s}\gtrsim_u \|\A\|^{s}.\end{equation}
\end{prop}
%\begin{remark}
%For (\ref{equiMeta}) the implicit constants depend only on $s$ and $n$, while
%for (\ref{optimallower-Meta}) the implicit constant also depends on $u$.
%%We remark that all the implicit contants don't depend on the parameter `t', even when $t\in \mathbb R$.
%\end{remark}

\begin{prop} \label{schandmeta}
Given $u\in \CH^s(\R^n)\setminus \{0\}$ with $s>0$, we have
\begin{equation}\label{esti-schro-meta-p}
\|\rho(v)\CM(\A)u\|_s\simeq_u (\|v\|+\|\A\|)^s, \qquad \forall \ v\in \R^{2n},\ \A\in \Sp(n,\R).\end{equation}
\end{prop}
\proof For $v=\left(\begin{array}{c}p\\q\end{array}\right)$ with $p,q\in\R^n$, in view of (\ref{SchroRep-L2}), we have
\begin{eqnarray*}
(\rho(v)\CM(\A) u)(x)&=&e^{\ri \la q,x \ra+\frac{\ri}{2} \la p, q\ra} (\CM(\A) u)(x+p),\\
(\CF\rho(v)\CM(\A) u)(\xi)&=&e^{\ri \la p,\xi \ra-\frac{\ri}{2} \la p, q\ra} (\CF\CM(\A) u)(\xi-q).
\end{eqnarray*}
A change of variables and the upper bound in (\ref{equiMeta}) give the estimates
\begin{eqnarray*}
\||X|^s\rho(v) \CM(\A)u\|_0^2
&=&\int_{\R^n}\|x-p\|^{2s}\left|(\CM(\A) u)(x)\right|^2 \, dx\\
&\lesssim&\|p\|^{2s}\|\CM(\A) u\|_0^2+\||X|^s\CM(\A) u\|_0^2\\
&\lesssim&\|p\|^{2s}\|u\|_0^2+\|\A\|^{2s}\|u\|^2_s,\\
\||X|^s\CF\rho(v) \CM(\A)u\|_0^2&\lesssim&\|q\|^{2s}\|u\|_0^2+\|\A\|^{2s}\|u\|^2_s.
\end{eqnarray*}
The above inequalities give the upper-bound part of (\ref{esti-schro-meta-p}).

Before showing the lower-bound part of (\ref{esti-schro-meta-p}), we first establish the following auxiliary estimate: 
for $Z=(D_1,\cdots,D_n,X_1,\cdots,X_n)^*$,
for given $u\in \CH^s(\R^n)$ with $s>0$, there exists $a_{u}>0$ 
such that
\begin{equation}\label{eq:observable}
 \||r\cdot Z-b|^su\|_{0}
 \geq a_{u}(\|r\|^2+b^2)^{\frac{s}{2}},\quad \forall \ (r,b)\in\R^{2n}\times\R.
\end{equation}
Define
$\Phi_{u}(r,b):=\||r\cdot Z-b|^su\|_{0}$.
This function is homogeneous of degree $s$, i.e.,
$$
 \Phi_{u}(\lambda r,\lambda b)
 =|\lambda|^s\Phi_{u}(r,b).
$$
The map $\Phi_u$ is continuous. Indeed, for every compact
$K\subset\R^{2n+1}$, the standard harmonic-oscillator estimates
and interpolation yield
$$
 \sup_{(r,b)\in K}\||r\cdot Z-b|^sf\|_0
 \lesssim_{K}\|f\|_s.
$$
Continuity is immediate for $f$ in the Schwartz space $\CS(\R^n)$; the estimate
above and the density of $\CS(\R^n)$ in $\CH^s(\R^n)$ then give
continuity for $u\in\CH^s(\R^n)$.

Now, it is sufficient to consider it on the compact sphere
$$
{\Bbb S}^{2n}
 =\{(r,b)\in\R^{2n+1}:\|r\|^2+b^2=1\}.
$$
If $r=0$, then $|b|=1$ and
$\Phi_{u}(0,b)=\|u\|_{0}>0$. If $r\neq0$ and
$\Phi_{u}(r,b)=0$, the spectral measure of $r\cdot Z$ associated
with $u$ is supported at $b$. Hence $u$ is an $L^2$ eigenfunction
of $r\cdot Z$. But every non-zero real linear combination of $D$ and $X$
has empty point spectrum, and consequently has no non-zero $L^2$ eigenfunction. Thus
$$
 \min_{\|r\|^2+b^2=1}\Phi_{u}(r,b)>0,
$$
which proves (\ref{eq:observable}) by homogeneity.

To show the lower-bound part of (\ref{esti-schro-meta-p}), it is sufficient to show that 
\begin{equation}\label{lower-pq}\||X|^s\rho(v) \CM(\A)u\|_0+\||X|^s\CF\rho(v) \CM(\A)u\|_0\gtrsim_u  (\|v\|+\|\A\|)^s.\end{equation}
By direct computations with (\ref{rhode}) and (\ref{hengdeng1}), we have
\begin{equation}\label{eq:covariance}
\CM(\A)^{-1}Z\CM(\A)=\A^{*-1}Z.
\end{equation}
Let $r_j(\A)$, $1\leq j\leq 2n$, be the rows of $\A^{*-1}$, i.e.,
$$\A^{*-1}=
 \begin{pmatrix}r_1(\A)\\ \vdots\\ r_{2n}(\A)\end{pmatrix},
 \qquad r_j(\A)^*\in\R^{2n}.$$
Applying (\ref{eq:observable}) and (\ref{eq:covariance}), we obtain
\begin{eqnarray*}
\||X|^s\rho(v) \CM(\A)u\|_0^2
&=&\int_{\R^n}\|x-p\|^{2s}\left|(\CM(\A) u)(x)\right|^2 \, dx\\
 &\gtrsim&\sum_{j=1}^n
       \||X_j-p_j|^s\CM(\A) u\|_{0}^2\\
 &=& \sum_{j=1}^{n}
       \||r_{n+j}(\A)\cdot Z-p_j|^su\|_{0}^2\\
 &\geq& a^2_{u}\sum_{j=1}^n
       \left(\|r_{n+j}(\A)\|^2+p_j^2\right)^s,\\
\||X|^s\CF\rho(v) \CM(\A)u\|_{L^2}^2 &\gtrsim&\sum_{j=1}^n\||D_j+q_j|^s\CM(\A) u\|_{0}^2 \\
&\gtrsim&  a^2_{u}\sum_{j=1}^n
       \left(\|r_{j}(\A)\|^2+q_j^2\right)^s. 
%\||X|^s\CF\rho(v) \CM(\A)(\,\cdot+q)\|_{L^2}^2
% &=&\int_{\R^n}|x|^{2s}|\widehat v(x+q)|^2\,dx\\
% &=&\int_{\R^n}|\xi-q|^{2s}|\widehat v(\xi)|^2\,d\xi\\
% &=&\||D-q|^sv\|_{L^2}^2.\\
%  \||X|^s\mathcal Fv(\,\cdot+q)\|_{L^2}^2
% &\geq& \gamma_{n,s}\sum_{j=1}^n
%       \||D_j-q_j|^sv\|_{L^2}^2\\
% &=&\gamma_{n,s}\sum_{j=1}^n
%       \||r_{n+j}(A)\cdot Z-q_j|^su\|_{L^2}^2\\
% &\geq& c_{s,u}\sum_{j=1}^n
%       \bigl(|r_{n+j}(A)|^2+q_j^2\bigr)^s.
\end{eqnarray*}
%Here we used the finite-dimensional inequality, valid for every
%$s>0$,
%\[
% \left(\sum_{j=1}^nx_j^2\right)^s
% \geq \gamma_{n,s}\sum_{j=1}^n|x_j|^{2s},
% \qquad
% \gamma_{n,s}:=\min\{1,n^{s-1}\}>0.
%\]
It follows that
\begin{eqnarray*}
 & &\||X|^s\rho(v) \CM(\A)u\|_{L^2}
 +\||X|^s\CF\rho(v) \CM(\A)u\|_{L^2}\nonumber\\
 &\gtrsim& a_u
 \left(\sum_{j=1}^{2n}\|r_{j}(\A)\|^2+\|p\|^2+\|q\|^2\right)^{\frac{s}2}\nonumber\\
  &\gtrsim& a_u
 \left(\|\A\|^2+\|p\|^2+\|q\|^2\right)^{\frac{s}2},
\end{eqnarray*}
which yields (\ref{lower-pq}).\qed

%Since $a+b\geq(a^2+b^2)^{1/2}$ for $a,b\geq0$, the left-hand side
%of (\ref{eq:main}) is bounded from below by
%\[
% C_{s,u}
% \left[
% \sum_{j=1}^n\bigl(|r_j(A)|^2+p_j^2\bigr)^s
% +\sum_{j=1}^n\bigl(|r_{n+j}(A)|^2+q_j^2\bigr)^s
% \right]^{1/2}.
%\]
%For every $s>0$ and all nonnegative $a_1,\ldots,a_m$,
%\[
% \sum_{k=1}^m a_k^s
% \geq \delta_{m,s}\left(\sum_{k=1}^m a_k\right)^s,
% \qquad
% \delta_{m,s}:=\min\{1,m^{1-s}\}>0.
%\]

%By finite-dimensional norm equivalence,
%\[
% \left(\|S_A\|_{\mathrm F}^2+|p|^2+|q|^2\right)^{1/2}
% \geq C_n\bigl(\|S_A\|_{\ell^1}+|p|+|q|\bigr).
%\]
%If $S_A=A$, then
%$\|S_A\|_{\ell^1}=\|A\|_{\ell^1}$. If $S_A=A^{-1}$, the
%symplectic identity
%\[
% A^{-1}=-JA^TJ
%\]
%shows that the entries of $A^{-1}$ are obtained from those of $A$ by
%permutations and changes of sign. Hence
%\[
%\|A^{-1}\|_{\ell^1}=\|A\|_{\ell^1}.
%\]
%Combining this with (\ref{eq:frobenius}) proves (\ref{eq:main}).

\subsection{Conjugation in classical and quantum Hamiltonians}\label{secODEtoPDE}

According to \cite[Lemmas 4.3 and 4.4]{LLZ2025}, we have
\begin{prop}\label{prop_MetaandSch}
If the affine system
\begin{equation*}\label{ODEhq}
\dot{z}= \CA(t) z + \ell(t),\quad \CA(\cdot)\in C_{b}^0(\R, \sp(n, \R)),\quad \ell(\cdot)\in C_{b}^0(\R, \R^{2n})
\end{equation*}
is conjugated to $\dot{y}= \J_n y$ via some transformation
$$z(t)=U(t) y(t)+ v(t),\quad  U(\cdot)\in  C^1(\R, \Sp(n,\R))\footnote{The notation $C^1
(\R,\bullet)$ with $\bullet=\Sp(n, \R)$ or $\R^{2n}$ is the subspace of $C^0
(\R,\bullet)$, in which the matrix or vector has entries $C^1$ w.r.t. $t$.},\quad v(\cdot)\in C^1(\mathbb R, \R^{2n}),$$
%where $v(t)$ is a solution of the equation (\ref{ODEhq}) and $U(t)$ satisfies
%\begin{eqnarray}\label{xinbianhuan}
%\dot{U}=\CA(t)U-U\CB.
%\end{eqnarray}
then, for the two Hamiltonian PDEs
$$
\frac{1}{{\rm i}} \partial_t \psi_1=\left(\CQ_{\CA(t)}(Z)+\CL_{\ell(t)}(Z)\right)\psi_1,\qquad \frac{1}{{\rm i}} \partial_t \psi_2=\CQ_{\J_n}(Z)\psi_2,
$$
their solutions satisfy
$$\psi_1(t)= e^{ {\rm i}\int_0^{t}\CC(\tau)d\tau} \rho(v(t)) \CM( U(t))  \psi_2(t),\quad \CC(t):=-\frac12\la v(t), \J_n
\ell(t)\ra.$$
\end{prop}
%\proof
%We start from the equation (\ref{ODEhq}). In the first step, we use the transformation $z=w+v(t)$. Since $w(t)$ is the solution of (\ref{ODEhq}), the original equation is reducible to
%$w'= \CA(t) w$. As the second step, we  use the symplectic transformation $w=U(t) y$. Since $U(t)$ satisfies
%(\ref{xinbianhuan}), one obtains $\dot{y}=\CB y$ by a direct computation.
%The following proof follows by Lemma \ref{reducibleprop}, Lemma \ref{reduciblepropnew} and Remark \ref{changshubianhuan}.
%\qed
%{\clo As an application of Proposition \ref{prop_MetaandSch}, we often choose $\CB=\mathbb J_n$.}

\section{Generalized reducibility - Proof of Theorem \ref{thm-reduc}}

According to Proposition \ref{prop_MetaandSch},
in order to establish the generalized reducibility, defined as in Theorem \ref{thm-reduc},
for the quadratic quantum Hamiltonian (\ref{orig-equ-1}), it is sufficient to show that the affine system
\begin{equation}\label{ODE-proof}
\dot{z}=\CA(t)z+\ell(t),\quad \CA(\cdot)\in C_b^{0}(\R, \sp(n,\R)),\quad \ell(\cdot)\in C_b^{0}(\R, \R^{2n}),
\end{equation}
is reducible in the generalized sense to the linear system $\dot{y}=\J_n y$.

\begin{prop}\label{prop_gen_red}
There exist $U\in C^1(\R, \Sp(n, \R))$ with $U(t_0) = \J_n$, and a particular solution $z_{*}\in C^1(\R, \R^{2n})$ of (\ref{ODE-proof}), such that, under the transformation $z(t) = U(t) y(t)+z_{*}(t)$,
the affine system (\ref{ODE-proof}) is conjugated to the linear system $\dot{y}=\J_n y $.
\end{prop}
\proof We first consider the linear system (\ref{ODE-proof}) with $\ell(\cdot)=0$.
To conjugate this system to $\dot{y}=\J_n y $, it suffices to find
$U\in C^1(\R,\Sp(n,\R))$ satisfying
\begin{equation}\label{Udeeqs}
\dot{U}(t)=\CA(t)U(t)-U(t)\J_n,
\qquad U(t_0)=\J_n.
\end{equation}
By linear ODE theory on each finite time interval,
the fundamental matrix $W$ exists on $[t_0,\infty[$, satisfying
$$
\dot{W}(t)=\CA(t)W(t),
\qquad W(t_0)=\I_{2n}.
$$
Since $\CA(t)\in \sp(n, \R)$, we have $\CA(t)^{*} \J_n + \J_n \CA(t)=0$.
It follows that 
$$\frac{d}{dt}W(t)^{*} \J_n W(t)=0.$$
Under the assumption $W(t_0)=\I_{2n}$, we have
$$W(t)^{*} \J_n W(t)=\J_n,\quad \forall \ t\geq t_0,$$
which implies that $W(t) \in \Sp(n,\R)$ for every $t\geq t_0$. 
Then we obtain $U\in C^1(\R,\Sp(n,\R))$ by defining
$$U(t):=W(t)\J_ne^{-(t-t_0)\J_n},$$ which clearly satisfies (\ref{Udeeqs}).

\indent We now return to the affine system (\ref{ODE-proof}) with non-vanishing
$\ell(\cdot)$. Since $z_{*}$ is a particular solution of (\ref{ODE-proof}),
the translation
$$
z(t)=w(t)+z_{*}(t)
$$
transforms the affine system into the linear system
$
\dot w=\CA(t)w$.
This linear system is further reduced to
$
\dot y=\J_n y
$
via the transformation $w(t)=U(t)y(t)$, where
$U(\cdot)\in C^1(\R,\Sp(n,\R))$ satisfies (\ref{Udeeqs}).\qed

%Combining Propositions \ref{prop_MetaandSch} and \ref{prop_gen_red}, Theorem \ref{thm-reduc} is shown through the following diagram:
%
%\noindent$$
%\begin{array}{rcccl}
% & {\rm Affine \ system}   & &    {\rm Hamiltonian \ PDE} &  \\
%  &   &  &  &   \\
%&z'=\CA(t)z+\ell(t) &\longleftrightarrow&  \frac{1}{{\rm i}} \partial_t \psi= (  \CQ_{\CA(t)} +  \CL_{\ell(t)} )\psi & \\
%  &   &  &  &   \\
%z=w+z_*(t) &\big\downarrow &  & \big\downarrow  &  \psi =\rho(z_*(t))  \phi  \\
%  &  &  &  &   \\
%  &w'=\CA(t) w   &\longleftrightarrow&  \frac{1}{{\rm i}}  \partial_t \phi=\CQ_{\CA(t)}\phi  &  \\
%    &  &  &  &   \\
%w=U(t) y & \big\downarrow &  & \big\downarrow    & \phi  =\CM(U(t))   \varphi \\
%  &  &  &  & \\
%&y'=\J_n y  &\longleftrightarrow&   \frac{1}{{\rm i}} \partial_t \varphi= \CQ_{\J_n}\varphi&
%  \end{array}
%$$
%\stackrel{\rm Quantized}{\longrightarrow}

\section{Growth rate of Sobolev norms}

With the generalized reducibility, we will investigate the long-time behavior of solutions to the corresponding quantum Hamiltonian in the Sobolev space.

\smallskip

\noindent{\it Proof of Theorem \ref{thmODEtoPDE}.} 
By Propositions \ref{prop_MetaandSch} and \ref{prop_gen_red}, the quantum Hamiltonian PDE
(\ref{orig-equ-2})
is transformed into
\begin{equation}\label{HsQHO}
 \frac{1}{{\rm i}} \partial_t \varphi= \CQ_{\J_n}\varphi = \CT \varphi
\end{equation}
under the $L^2-$unitary transformation $\psi(t)=\rho(z_{*}(t))\CM(U(t)) \varphi(t)$, where a scalar factor of modulus one is omitted.

%It is well known that the solution $\varphi(t)$ to Eq. (\ref{HsQHO}) satisfies
%\begin{equation}\label{bddHsQHO}
%\|\varphi(t)\|_s\simeq \|\varphi(t_0)\|_s,\quad s\geq 0.
%\end{equation}

Let $W(t)$ be the fundamental solution matrix of the linear system $\dot{w}=\CA(t)w$ with
$W(t_0) = \I_{2n}$. Through direct computations, we have
\begin{align}\label{Wtex}
W(t) = U(t) \exp((t-t_0) \J_n) \J_n^{*},
\end{align}
with $U(t)$ in Proposition \ref{prop_gen_red}.
Then, combining  (\ref{Wtex}) with Proposition \ref{schandmeta}, we obtain (\ref{esti_Hs}), since
%\begin{eqnarray*}
%\|\psi(t)\|_s
%&=&\| \rho(z_{*}(t))\CM(U(t)) \varphi(t) \|_s \nonumber\\
%&\simeq&\|\varphi(t_0)\|_s\\
%&=&\| \rho(z_{*}(t))\CM(U(t)) \CM(\J_n^*) \rho(-z_{*}(t_0)) \psi_0 \|_s. \end{eqnarray*}
\begin{align*}
\|\psi(t)\|_s
&=\| \rho(z_{*}(t))\CM(U(t)) \varphi(t) \|_s    \\
& \simeq \|\rho(z_{*}(t)) \CM(W(t)) \rho(-z_{*}(t_0))\psi(t_0)\|_s\\
& \simeq \|z_{*}(t)\|^s +\|W(t)\|^s,
\end{align*}
where the implicit constant depends on $z_{*}(t_0)$ and $\psi(t_0)$.
In the case $\ell(\cdot)=0$, we may choose $z_*\equiv0$. Since $\rho(0)=\Id$, we have $\|\psi(t)\|_s\simeq \|W(t)\|^s$. \qed

\medskip

By Theorem \ref{thmODEtoPDE}, we see that in order to study the solution of the quadratic quantum Hamiltonian (\ref{orig-equ-2}), it is in fact sufficient to analyze the solutions of the corresponding classical Hamiltonian systems (\ref{orig-ODE-2}), regardless of whether the latter are reducible in the classical sense.
To prove Theorems~\ref{thmGR-mono} and~\ref{thmGR-osci}, we therefore establish the following abstract result for homogeneous quadratic quantum Hamiltonians. 

For simplicity, we denote
\begin{eqnarray}
H_f(t)&:=& \exp\left\{\int_{t_0}^t  q(s)\cos^2(s) \, ds \right\} \ {\rm with } \  q(t): = \frac{f'(t)}{f(t)} \label{Hft}\\
 K(t)&:=&\int_{t_0}^{t}\frac{\sin(2s) \, q(s)}{H_{f}(s)^2} \, ds. \label{eq:K}
\end{eqnarray}
We write $H(t)$ for $H_{f}(t)$ when $f$ is clear from context.

 Given $f\in C^2([t_0,\infty[,\R^*_+)$, assume that
\begin{eqnarray}
&&\inf\limits_{t\geq t_0}f(t)>0, \label{condf0}\\
& &  \sup\limits_{t\geq t_0} |q(t)|<\infty, \label{condf1}\\
& &\sup_{t\geq t_0}\left|\int_{t_0}^t  q(s)\cos(2s) \, ds\right|<\infty,\label{condf2}\\
& &\sup_{t\geq t_0}\left|K(t)\right|<\infty. \label{condf3}
\end{eqnarray}
Throughout the remainder of this section, all implicit constants appearing in estimates denoted by ``$\lesssim$", ``$\gtrsim$" and ``$\simeq$" are allowed to depend on the given growth rate function $f$ and the initial moment $t_0$.

\begin{thm}\label{thmGr}
Assume that $f\in C^2([t_0,\infty[,\R^*_+)$ satisfies (\ref{condf0}) -- (\ref{condf3}). There exists $\phi_f\in C^0([t_0,\infty[,\R)$ such that, for $s>0$ and a non-vanishing initial condition $u(t_0)\in \CH^s$, there exists a constant $c_0>1$ such that (\ref{growth}) holds for
the solution $u=u(t)$ to Eq. (\ref{QHe-abs}).
\end{thm}

\begin{remark}
Theorem \ref{thmGr} also applies to $f(t)=e^{\lambda t}$ for $\lambda>0$, since its hypotheses hold with $q(t)=\lambda$. This growth rate has already been observed in \cite{MR2017}.
\end{remark}

\proof With $f$ satisfying the conditions (\ref{condf1}) -- (\ref{condf3}), we define the function $\phi_f$ explicitly as
\begin{equation}\label{phif}
\phi_f(t):=-  q^2(t)\cos^4(t)-q'(t)\cos^2(t)
+  4q(t)\cos(t)\sin(t). 
\end{equation}
For the homogeneous quadratic quantum Hamiltonian
$$
\frac1\ri\partial_t u=\frac12\left(D^2 + (1+\phi_f(t))X^2\right) u,
$$ it is sufficient to consider the non-autonomous linear system
\begin{equation}\label{or1ODEsys}\left(\begin{array}{c}
\xi\\
x
\end{array}\right)'=\left(\begin{array}{cc}
0 &  1\\
-(1+\phi_f(t)) & 0
\end{array}\right)\left(\begin{array}{c}
\xi\\
x
\end{array}\right).\end{equation}
Since $\xi'=x$, the system is equivalent to the second-order equation
\begin{equation}\label{or2ODE}
\xi''+(1+\phi_f (t))\xi=0.
\end{equation}

Through direct computations, we can verify that Eq. (\ref{or2ODE}) admits two linearly independent solutions
\begin{equation}\label{x_12}
  \xi_1(t)=\cos(t)H_f(t),\quad
  \xi_2(t)=\frac{\sin(t)}{H_f(t)} +  \xi_1(t)K(t).
\end{equation}
Set $x_j(t)=\xi'_j(t)$ for $j=1,2$. Direct differentiation gives
\begin{eqnarray}
x_1(t)&=&\left(-\sin(t)+q(t)\cos^3(t)\right)H_f(t),\label{xi_1}\\
x_2(t)&=&\left(\cos(t)+q(t)\cos^2(t)\sin(t)\right)\frac{1}{H_f(t)}\label{xi_2}\\
 & & +  \, \left(q(t)\cos^3(t)-\sin(t)\right)H_f(t)K(t). \nonumber
\end{eqnarray}
The Wronskian is $\xi_1(t)\xi_2'(t)-\xi'_1(t)\xi_2(t)=1$.
Consequently, for the system (\ref{or1ODEsys}), we obtain a fundamental matrix
$$
 \widetilde W(t):=
 \begin{pmatrix}\xi_2(t)&-\xi_1(t)\\x_2(t)&-x_1(t)\end{pmatrix}  \quad {\rm with} \quad  \det \widetilde W(t)=1,
$$
and, by defining the time-dependent matrices
\begin{eqnarray*}
S(t)&:=&\left(\begin{array}{cc} 1 & 0 \\q(t)\cos^2(t) & 1\end{array}\right),\\
R(t)&:=& \left(\begin{array}{cc} \sin (t) &   -\cos (t) \\ \cos(t) & \sin(t)\end{array}\right),\\ 
L(t)&:=&\left(\begin{array}{cc} H_{f}^{-1}(t) & 0 \\ -H_{f}(t)K(t) & H_{f}(t)\end{array}\right),
\end{eqnarray*}
 we have the decomposition $\widetilde W(t)= S(t) R(t) L(t)$.

Let us show that 
\begin{equation}\label{equi-WHf}
\|\widetilde W(t)\| \simeq |H_{f}(t)|+|H_{f}^{-1}(t)|,\quad \forall \ t\geq t_0. 
\end{equation}
It is easy to verify that
 \begin{equation}\label{condRandRni}
\sup\limits_{t\geq t_0} \left(\|R(t)\|+\|R^{-1}(t)\|\right)\leq 4.
 \end{equation}
Under assumption (\ref{condf1}), we also have
\begin{equation}\label{condSandSni}
\sup\limits_{t\geq t_0} \left(\|S(t)\|+\|S(t)^{-1}\|\right)<\infty.
\end{equation}
Combining (\ref{condRandRni}) and (\ref{condSandSni}), we obtain the lower bound in (\ref{equi-WHf}) by
$$\|\widetilde W(t)\| \simeq \|L(t)\| \simeq |H_{f}(t)|+|H_{f}^{-1}(t)|+|H_{f}(t)K(t)|\geq |H_{f}(t)|+|H_{f}^{-1}(t)|.$$
On the other hand, under the assumption (\ref{condf3}), we have $|K(t)| \lesssim 1$, which implies the upper bound 
$$\|\widetilde W(t)\| \simeq \|L(t)\|  \lesssim |H_{f}(t)|+|H_{f}^{-1}(t)|. $$

According to Lemma \ref{lemmafulu2}, the assumption (\ref{condf2}) implies 
$$
H_f(t)\simeq f(t)^\frac12,\quad \forall \ t\geq t_0.
$$
Hence, together with (\ref{condf0}), the relation (\ref{equi-WHf}) implies that
$$\|\widetilde W(t)\|\simeq  |H_{f}(t)|+|H_{f}^{-1}(t)| \simeq f^{\frac12}(t) +f^{-\frac12}(t)\simeq f^{\frac12}(t), \quad  \forall \  t\geq t_0.  $$
Normalizing the fundamental matrix by $W(t):=\widetilde W(t)\widetilde W(t_0)^{-1}$, we have 
$W(t_0)=\I_2$, and $W(t)=\widetilde W(t) R^{-1}(t_0) S^{-1}(t_0)$.
Through (\ref{condRandRni}) and (\ref{condSandSni}), we obtain
\begin{eqnarray}\label{jijieguji}
\|W(t)\|\simeq \|\widetilde W(t) \|\simeq f^{\frac12}(t),\quad \forall \ t\geq t_0. 
\end{eqnarray}
Through Theorem \ref{thmODEtoPDE}, we obtain (\ref{growth}) for $f\in C^2([t_0,\infty[,\R^*_+)$ satisfying (\ref{condf0}) -- (\ref{condf3}).\qed

\subsection{Monotone sub-exponential growth}\label{sec-GRmono}
Let us prove Theorem \ref{thmGR-mono} for $f\in \CM([t_0,\infty[)$ (defined by (\ref{Mclass1}) -- (\ref{Mclass3})). 

\subsubsection{Proof of Theorem \ref{thmGR-mono}-(\ref{thmGR-mono-1})} We first show that the assumptions (\ref{condf0}) -- (\ref{condf3}) hold for any $f\in \CM([t_0,\infty[)$. The estimates (\ref{growth}) for the homogeneous equation (\ref{QHe-abs}) then follow from Theorem \ref{thmGr}. 

\smallskip

%\noindent{\it Verification of assumptions (\ref{condf0}) -- (\ref{condf3}).} 

The assumption (\ref{condf0}) is exactly (\ref{Mclass1}) in the definition of $\CM([t_0,\infty[)$. The assumption (\ref{condf1}) is satisfied since (\ref{Mclass3}) implies that $q(t)$ tends to $0$ as $t\to\infty$.

For $f\in \CM([t_0,\infty[)$, (\ref{Mclass2}) implies that $q=f'/f \in C^1([t_0,\infty[,\R_+)$.
Since (\ref{Mclass3}) implies that $q(t)\to 0$ with $q'\leq0$,
we have
$$
 \int_{t_0}^{\infty}|q'(s)| \, ds=q(t_0)<\infty.
$$
Integration by parts gives
$$\int_{t_0}^t  q(s)\cos(2s) \, ds =\frac12[q(s)\sin(2s)]_{t_0}^{t}
     -\frac12\int_{t_0}^{t}q'(s)\sin(2s) \, ds,$$
which implies, through (\ref{Mclass2}), that
\begin{equation}\label{bddqcos}
 \left|\int_{t_0}^t  q(s)\cos(2s) \, ds\right|
\leq\frac12q(t)+\frac12q(t_0)
      +\frac12\int_{t_0}^{t}|q'(s)| \,  ds=q(t_0).\end{equation}
Thus the assumption (\ref{condf2}) is satisfied.

In view of (\ref{bddqcos}), the assumption (\ref{bddOscInt}) of Lemma \ref{lemmafulu2} is satisfied with $M=q(t_0)$. Hence,
\begin{equation}\label{equi-Hf+}
\frac{e^{-\frac{q(t_0)}2}}{\sqrt{f(t_0)}}f(t)^\frac12 \leq H_f(t) \leq \frac{e^{\frac{q(t_0)}2}}{\sqrt{f(t_0)}}f(t)^\frac12, \quad \forall \ t\geq t_0.\end{equation}
Since $q=f'/f \in C^1([t_0,\infty[,\R_+)$, we have
$$\left|\int_{t_0}^t \frac{\sin(2s)}{H_f(s)^2}q(s)\, ds\right| \leq \int_{t_0}^t \frac{q(s)}{H_f(s)^2}\, ds
 \lesssim\int_{t_0}^t\frac{f'(s)}{f(s)^2}\, ds \leq \frac{1}{f(t_0)}.$$
Thus the assumption (\ref{condf3}) is satisfied.

If $f''/f$ tends to $0$ as $t\to\infty$, then, combining with (\ref{Mclass3}), we have
$$q'(t)=\frac{f''(t)f(t)-f'(t)^2}{f(t)^2}\to 0,\quad  t\to\infty,$$
and hence $\phi_f(t) \to 0$. Consequently, (\ref{decaying-phi}) is shown.

\subsubsection{Proof of Theorem \ref{thmGR-mono}-(\ref{thmGR-mono-2})} For $f\in \CM([t_0,\infty[)$, the additional assumption (\ref{condisupp1}) gives
$$
 \log\frac{f(t)}{f(r)}
 =\int_r^tq(\sigma) \, d \sigma
 \leq\kappa\int_r^t\frac{d \sigma}{\sigma}
 =\kappa\log\frac tr.
$$
Exponentiating yields
\begin{equation}\label{ineq-f} 
f(r)\geq f(t)\left(\frac rt\right)^\kappa,
 \qquad \forall \  r\in [t_0, t].\end{equation}
With $r=t_0$, one obtains
$$f(t)\leq f(t_0)\left(\frac{t}{t_0}\right)^{\kappa}, \quad \forall \ t\geq t_0.$$
Since $0<\kappa<2$, it follows that
\begin{equation}\label{Hf-sublinear}
f(t)= O(t^{\kappa}), \qquad H(t)= O(t^{\frac\kappa2})=o(t),\quad {\rm as} \ t\to \infty.
\end{equation}

Let us show (\ref{growth-lin}) for Eq. (\ref{QHe-abs-nonhomo}).
According to Theorem \ref{thmODEtoPDE},
it is sufficient to consider the affine system
\begin{equation}\label{or1ODEsys-nonhomo}
\left(\begin{array}{c}
\xi\\
x
\end{array}\right)'=\left(\begin{array}{cc}
0 &  1\\
-(1+\phi_f(t)) & 0
\end{array}\right)\left(\begin{array}{c}
\xi\\
x
\end{array}\right)+a\left(\begin{array}{c}
0\\
\sin(t)
\end{array}\right).\end{equation}
For the corresponding linear system (\ref{or1ODEsys}), 
the two solutions in (\ref{x_12}) -- (\ref{xi_2}) give the fundamental matrix
\begin{equation}\label{eq:Wtilde}
\widetilde W(t)=\left(\begin{array}{cc}W_{11}(t) & W_{12}(t) \\W_{21}(t) & W_{22}(t)\end{array}\right)
=\left(\begin{array}{cc}\xi_2(t)&- \xi_1(t)\\ x_2(t)& -x_1(t)\end{array}\right).
\end{equation}
The Wronskian identity gives
\begin{eqnarray}\label{hanglie1}
\det\widetilde W(t) =1, \quad \forall  \ t\geq t_0. 
\end{eqnarray}
Let $W(t):=\widetilde W(t)\widetilde W(t_0)^{-1}$. Then, $W(t_0)=\I_{2}$, 
%$W(t)\in{\rm SL}(2,\R)$,
and as (\ref{jijieguji}),
\begin{equation}\label{jijieguji-M}
\|W(t)\|\simeq f^{\frac12}(t), \quad \forall  \ t\geq t_0. \end{equation}
Therefore, for the affine system (\ref{or1ODEsys-nonhomo}) with initial condition $z(t_0)$, we have the particular solution
\begin{eqnarray}
z_{*}(t)&=&W(t)z(t_0)+a W(t) \int_{t_0}^t W(s)^{-1} \left(\begin{array}{c}
0\\
\sin (s)
\end{array}\right) \, ds\label{z*}\\
&=& \widetilde W(t)\widetilde W(t_0)^{-1} z(t_0)+a  \widetilde W(t) \int_{t_0}^t  \widetilde W(s)^{-1} \left(\begin{array}{c}
0\\
\sin (s)
\end{array}\right) \, ds.\nonumber
\end{eqnarray}
Let us denote 
$$z_p(t):= a  \widetilde W(t) \int_{t_0}^t  \widetilde W(s)^{-1} \left(\begin{array}{c}
0\\
\sin (s)
\end{array}\right) \, ds. $$
%\begin{eqnarray*}
%z_p(t): &=& a W(t) \int_{t_0}^t W(s)^{-1} \left(\begin{array}{c}
%0\\
%\sin (s)
%\end{array}\right) \, ds\\ 
%&=& a  \widetilde W(t) \int_{t_0}^t  \widetilde W(s)^{-1} \left(\begin{array}{c}
%0\\
%\sin (s)
%\end{array}\right) \, ds.
%\end{eqnarray*}
According to (\ref{eq:Wtilde}) and (\ref{hanglie1}), we have
$$
 \widetilde W(s)^{-1}\binom01=\binom{\xi_1(s)}{\xi_2(s)}.
$$
It follows that 
\begin{equation}
 \frac{z_p(t)}a=
\left(\begin{array}{l}
\xi_2(t)\displaystyle\int_{t_0}^{t}\sin(s) \xi_1(s)\, ds
 -\xi_1(t)\displaystyle\int_{t_0}^{t}\sin(s) \xi_2(s)\, ds\\
 \xi'_2(t)\displaystyle\int_{t_0}^{t}\sin(s) \xi_1(s)\, ds
 -\xi'_1(t)\displaystyle\int_{t_0}^{t}\sin(s) \xi_2(s)\, ds
\end{array}\right).
 \label{eq:zpexactfirst}
\end{equation}

For simplicity we denote 
\begin{eqnarray}
 B(t)&:=& \int_{t_0}^{t}\sin(2r) H(r)dr \label{eq:B},\\
 D(t) &:=& \int_{t_0}^{t}\sin(2r) H(r)
 \left(\int_r^t\frac{\sin(2\sigma) q(\sigma)}{H(\sigma)^2}
 d\sigma\right)dr,  \label{eq:D} \\
 I(t) &:=& \int_{t_0}^t\frac{\sin^2(s)}{H(s)}ds. \label{eq: I} 
\end{eqnarray}
The scalar integrals in (\ref{eq:zpexactfirst}) satisfy
\begin{eqnarray*}
\int_{t_0}^{t}\sin(s) \xi_1(s)\, ds&=&\frac12B(t),\\
\int_{t_0}^{t}\sin(s) \xi_2(s) \, ds&=&I(t)+\frac12\int_{t_0}^{t}\sin(2s)H(s)K(s)\, ds.
\end{eqnarray*}
By the definitions of $D(t)$ and $K(t)$ (recalling (\ref{eq:K})),
\begin{eqnarray*}
 D(t)
 &=&\int_{t_0}^{t}\sin(2s)H(s)\left(K(t)-K(s)\right)\, ds\\
 &=&K(t)B(t)-\int_{t_0}^{t}\sin(2s) H(s)K(s)\, ds.
\end{eqnarray*}
Therefore,
\begin{equation*}
 \int_{t_0}^{t}\sin (s) \xi_2(s)\, ds
 =I(t)+\frac12\bigl(K(t)B(t)-D(t)\bigr).
\end{equation*}
Set
$
 P(t):=H(t)I(t).
$
Substitution into (\ref{eq:zpexactfirst}) gives
\begin{eqnarray*}
 \frac{(z_p)_1(t)}a
 &=&- \, \xi_1(t)I(t)+\frac12\left(\xi_2(t)-\xi_1(t)K(t)\right)B(t)
 +\frac12\xi_1(t)D(t),\\
 \frac{(z_p)_2(t)}a
 &=&- \, \xi'_1(t)I(t)+\frac12\left(\xi'_2(t)-\xi'_1(t)K(t)\right)B(t)
 +\frac12\xi'_1(t)D(t).
\end{eqnarray*}
A straightforward computation yields the exact decomposition
$$
 \frac{z_p(t)}a=P(t)\binom{-\cos(t)}{\sin(t)}
 +\binom{R_1(t)}{R_2(t)},
$$
with the remaining terms defined as
\begin{eqnarray*}
 R_1(t)&:=&\frac{\sin(t)}{2H(t)}B(t)+\frac12\cos(t) H(t)D(t),\\
  R_2(t)&:=&-q(t)\cos^3(t) P(t) + \frac{\cos(t)+q(t)\cos^2(t)\sin(t)}{2H(t)}B(t)\\
 & &+ \, \frac12\left(-\sin(t)+q(t)\cos^3(t)\right)H(t)D(t).
\end{eqnarray*}

Through Lemma \ref{lemmaa5}, we obtain
$$I(t)\leq \int_{t_0}^{t}\frac{ dr}{H(r)},\qquad \forall \ t\geq t_0.$$  
Combining (\ref{equi-Hf+}) and (\ref{eq:I0}) in Lemma \ref{lemmaA3}, we have 
$$0\leq P(t)\lesssim t-t_0,\quad \forall \ t\geq t_0.$$
Hence, with (\ref{condisupp1}),
$$
 q(t)P(t)\lesssim \frac\kappa t (t-t_0)\lesssim 1,\quad \forall \ t\geq t_0.
$$
This uniformly bounds the first term of $R_2(t)$.
By Lemma \ref{lemma_app_inez} and (\ref{equi-Hf+}), we obtain
$|B(t)|\lesssim H(t)$ and $|D(t)|\lesssim 1$. Then
\begin{equation}\label{eq:R1bound}
 |R_1(t)|\leq\frac12\frac{|B(t)|}{H(t)}+\frac12|H(t) |D(t) \lesssim H(t),\quad \forall \ t\geq t_0.
\end{equation}
Noting that (\ref{condisupp1}) implies $q(t)\leq \frac{\kappa}{t_0}\lesssim 1$, we have
\begin{equation}\label{eq:R2bound}
 |R_2(t)| \lesssim1+\frac{|B(t)|}{H(t)} +H(t)|D(t)|\lesssim H(t),\quad \forall \ t\geq t_0.
\end{equation}
Combining (\ref{Hf-sublinear}), (\ref{eq:R1bound}) and (\ref{eq:R2bound}), we have
$$
\left\| \binom{R_1(t)}{R_2(t)}\right\|\lesssim H(t)=o(t),\quad {\rm as} \ t\to \infty.$$

With (\ref{equi-Hf+}) and Lemma \ref{lemmaa5}, there exists $T_2> t_0$
such that
$$\left\|P(t)\binom{-\cos(t)}{\sin(t)}\right\| \simeq H(t) I(t) \simeq t,\quad \forall \ t\geq T_2. $$ 
Therefore, $R_1$ and $R_2$ are negligible as $t\to\infty$, and
\begin{equation}\label{gr_lin}
\left\|z_p(t)\right\| \simeq |a|\left\|P(t)\binom{-\cos(t)}{\sin(t)}\right\|\simeq t,\quad \forall \ t\geq T_2.
\end{equation}
It follows from (\ref{Hf-sublinear}) and (\ref{jijieguji-M}) that
$$\|W(t)z(t_0)\| \lesssim  f(t)^\frac12=O(t^{\frac\kappa2}),\quad \forall \ t\geq T_2 .$$ 
Thus, recalling (\ref{z*}), we have
\begin{equation}\label{zpt}
\left\|z_*(t)\right\| \simeq t,\quad \forall \ t\geq T_2.
\end{equation}
Since (\ref{jijieguji-M}) means that $\|W(t)\|\gtrsim f(t)^\frac12>0$ for $t\geq t_0 $, we have
$$\|z_*(t)\|^s+ \|W(t)\|^s\simeq 1,\quad \forall \ t\in [t_0, T_2].$$
Combining this estimate with (\ref{jijieguji-M}) and (\ref{zpt}), one obtains
 $$\|z_*(t)\|^s+ \|W(t)\|^s\simeq t^s,\quad \forall \  t\geq t_0, $$
  which implies (\ref{growth-lin}) through Theorem \ref{thmODEtoPDE}.

\subsection{Oscillatory sub-exponential growth}\label{sec-GRosci}

To show Theorem \ref{thmGR-osci}, we first define the following class of functions.
For a given $N\in\N^*$, let $\CA_N([t_0,\infty[)$ be the class of $C^N$ functions $G$ on $[t_0,\infty[$ with derivatives satisfying $G^{(N)}\in L^1([t_0,\infty[)$ and
$$G^{(j)}(t)\to 0 \  {\rm as} \ t\to\infty, \quad 0\leq j\leq N-1.$$

\begin{lemma}\label{lem-osci-int}
For $G\in\CA_N([t_0,\infty[)$ with some $N\in\N^*$, we have
$$\sup_{t\geq t_0}\left\{\left|\int_{t_0}^t \cos(2s) G(s) \, ds \right|,\ \left|\int_{t_0}^t \sin(2s) G(s) \, ds \right|\right\}<\infty. $$
\end{lemma}
\proof
By integration by parts, we have
\begin{eqnarray*}\int_{t_0}^t\cos(2s) G(s) \, ds
&=&\frac12\left[\sin(2s) G(s)\right]_{t_0}^t -  \frac12\int_{t_0}^t \sin(2s) G'(s) \, ds\label{int1_it_1}\\
&=&\frac12\left[\cos\left(2s-\frac\pi2\right)G(s)\right]_{t_0}^t \\
& & - \,  \frac12\int_{t_0}^t \cos\left(2s-\frac\pi2\right) G'(s) \, ds.
\end{eqnarray*}
We prove by induction that, for any $j\in\N^*$,
\begin{eqnarray*}
\int_{t_0}^t \cos(2s) G(s) \, ds&=&\sum_{l=1}^j\frac{(-1)^{l-1}}{2^l}\left[\cos\left(2s-\frac{l\pi}2\right)G^{(l-1)}(s)\right]_{t_0}^t\\
& & + \,  \frac{(-1)^j}{2^j}\int_{t_0}^t \cos\left(2s-\frac{j\pi}2\right)G^{(j)}(s) \, ds.\nonumber
\end{eqnarray*}
Assume that the above equality holds for $j=k\in\N^*$.
Through integration by parts,
\begin{eqnarray*}
& &\frac{(-1)^k}{2^k}\int_{t_0}^t \cos\left(2s-\frac{k\pi}2\right) G^{(k)}(s) \, ds\\
&=&\frac{(-1)^k}{2^{k+1}}\left[\cos\left(2s-\frac{(k+1)\pi}2\right)G^{(k)}(s)\right]_{s=t_0}^t \\
& & + \, \frac{(-1)^{k+1}}{2^{k+1}}\int_{t_0}^t \cos\left(2s-\frac{(k+1)\pi}2\right)G^{(k+1)}(s) \, ds,
\end{eqnarray*}
then we have that it also holds for $j=k+1$, and hence for every $j\in\N^*$ with $j\leq N$.
In particular, for $j=N$,
\begin{eqnarray*}
\int_{t_0}^t \cos(2s) G(s) \, ds&=&\sum_{l=1}^N\frac{(-1)^{l-1}}{2^l}\left[\cos\left(2s-\frac{l\pi}2\right)G^{(l-1)}(s)\right]_{t_0}^t  \\
& & + \, \frac{(-1)^{N}}{2^{N}}\int_{t_0}^t \cos\left(2s-\frac{N\pi}2\right)G^{(N)}(s) \, ds.
\end{eqnarray*}
Since $G\in \CA_N([t_0,\infty[)$, $G^{(N)}\in L^1([t_0,\infty[)$, we have that
$$\sup_{t\geq t_0}\left|\int_{t_0}^t \cos(2s) G(s) \, ds \right|<\infty. $$
Similarly,
$\displaystyle \sup_{t\geq t_0}\left|\int_{t_0}^t \sin(2s) G(s) \, ds \right|<\infty$.\qed

\medskip

%{\clr For $f$ given in (\ref{oscillatory_f}), with $\phi_f$ defined as in (\ref{phif}), by direct computations, we verify that (\ref{decay-oscillatory_f}) holds. With $H_f$ defined as in (\ref{Hft}), it is sufficient to show that,
%$$  \frac{f'}{f}\in \CA_2([t_0,\infty[), \quad \frac{f'}{H_f^2\cdot f}\in \CA_1([t_0,\infty[)$$
%Indeed, it can be verified by integration by parts. Hence $f$ satisfies (\ref{condf0}) -- (\ref{condf3}).}

\noindent{\it Proof of Theorem~\ref{thmGR-osci}.}
Write $f(t)=t^{\frac13}g(t)$ with 
$$g(t):=1+\log(t)\sin^2(\sqrt t).$$
Then $g\geq1$ for $t\geq e$, so \eqref{condf0} holds for every $t_0\geq e$.
Differentiating gives $g^{(j)}=O(t^{-\frac{j}{2}}\log(t))$ for $1\leq j\leq3$.
Using the fact that $g\geq1$ and
$$q(t)=\frac{1}{3t}+\frac{g'(t)}{g(t)},$$
we obtain
$$q(t)=O(t^{-\frac12}\log(t)),\quad  q'(t)=O(t^{-1}\log(t)^2),\quad q''(t)=O(t^{-\frac32}\log(t)^3).$$
In particular, $q$ is bounded and belongs to $\CA_2([t_0,\infty[)$.
Lemma~\ref{lem-osci-int} gives \eqref{condf2}, and
Lemma~\ref{lemmafulu2} then yields $H_f^2\simeq f$. Since
$f(t)\geq t^{\frac13}$, we have $q/H_f^2\to0$. Using
$H_f'(t)=q(t)\cos^2(t) H_f(t)$ gives
$$
 \left(\frac{q}{H_f^2}\right)'(t)
 =\frac{q'(t)-2q(t)^2\cos^2(t)}{H_f(t)^2}
 =O(t^{-\frac43}\log(t)^2).$$
This derivative is integrable, so $q/H_f^2\in\CA_1([t_0,\infty[)$.
Applying Lemma~\ref{lem-osci-int} to $q/H_f^2$ proves \eqref{condf3}.

The coefficient $\phi_f$ in \eqref{phif} is smooth on $[e,\infty[$ and
depends only on $q$ and $q'$. The bounds on $q$ and $q'$ give
$\phi_f(t)\to0$. Thus Theorem~\ref{thmGr} gives \eqref{growth} with
the same $\phi_f$ for every $t_0\geq e$. Finally, $g\geq1$ with
equality along $t=(k\pi)^2$, and $g(t)/\log(t)\leq1+1/\log(t)$ with
$g(t)/\log(t)\to1$ along $t=(k\pi+\pi/2)^2$. These bounds and the two
sequences give the claimed $\liminf$ and $\limsup$, respectively.\qed

\appendix

\section{Lemmas on oscillatory integrals}\label{secapp}
%\begin{lemma}\label{fulu1}
%Given the function $f\in \CM([t_0,\infty[)$ (defined in (\ref{Mclass1})--(\ref{Mclass3})), let $q=f'/f$. For the integral defined in (\ref{Hft}), we have
%$$\frac{e^{-\frac{q(t_0)}2}}{\sqrt{f(t_0)}} \sqrt{f(t)} \leq H_f(t)\leq \frac{e^{\frac{q(t_0)}2}}{\sqrt{f(t_0)}} \sqrt{f(t)},\quad \forall \ t\geq t_0.$$
%%where 
%%\[
%% h_-:=\frac{e^{-q(t_0)/2}}{\sqrt{f(t_0)}},
%% \qquad h_+:=\frac{e^{q(t_0)/2}}{\sqrt{f(t_0)}}.
%%\]
%\end{lemma}
%\proof
%Using the fact $\cos^2(r)=(1+\cos(2r))/2 $, we obtain the identity
%\begin{equation*}
% \log H_f(t)=\frac12\log\frac{f(t)}{f(t_0)}+\frac12E(t),
% \qquad E(t):=\int_{t_0}^{t}q(r)\cos(2r) \, dr.
%\end{equation*}
%Since $q\in C^1$ and the assumption (\ref{Mclass2}) implies that $q'\leq0$ and $q(t)\to 0$,
%we have
%$$
% \int_{t_0}^{\infty}|q'(r)| \, dr=q(t_0)<\infty.
%$$
%Integration by parts gives
%$$E(t)=\frac12[q(r)\sin(2r)]_{t_0}^{t}
%     -\frac12\int_{t_0}^{t}q'(r)\sin(2r) \, dr,$$
%which implies, through (\ref{Mclass2}),
%$$
% |E(t)|
%\leq\frac12q(t)+\frac12q(t_0)
%      +\frac12\int_{t_0}^{t}|q'(r)| \,  dr=q(t_0).
%$$
%Consequently,
%$$
%  H_f(t)=\sqrt{\frac{f(t)}{f(t_0)}}e^{\frac{E(t)}2},$$
% and the conclusion follows by a direct computation.  
% \qed

\smallskip

\begin{lemma}\label{lemmafulu2}
For $f\in C^2([t_0,\infty[,\R^*_+)$, if there exists $M>0$ such that
\begin{equation}\label{bddOscInt}
\sup_{t\geq t_0}\left|\int_{t_0}^t  q(s)\cos(2s) \, ds\right|\leq M \quad \it for \  q(t)=\frac{f'(t)}{f(t)},\end{equation}
then, for the integral $H(t)=H_f(t)$ defined in (\ref{Hft}), we have
$$\frac{e^{-\frac{M}2}}{\sqrt{f(t_0)}}\sqrt{f(t)} \leq H(t) \leq \frac{e^{\frac{M}2}}{\sqrt{f(t_0)}}\sqrt{f(t)}, \quad \forall \ t\geq t_0.$$
\end{lemma}
\proof
Using the identity $\cos^2(r)=(1+\cos(2r))/2 $, we obtain
\begin{equation*}
 \log H_f(t)=\frac12\log\frac{f(t)}{f(t_0)}+\frac12\int_{t_0}^{t}q(r)\cos(2r) \, dr.
\end{equation*}
The conclusion follows by a direct computation with (\ref{bddOscInt}).  
 \qed

\begin{lemma}\label{lemmaA3}
For $f\in\CM([t_0,\infty[)$ (defined in (\ref{Mclass1}) -- (\ref{Mclass3})) satisfying (\ref{condisupp1}), we have, for every $t\geq t_0$,
\begin{equation} \label{eq:I0}
 \sqrt{f(t_0)} e^{-\frac{q(t_0)}2}\frac{t-t_0}{\sqrt{f(t)}}
 \leq \CI_0(t):=\int_{t_0}^{t}\frac{ dr}{H(r)}
 \leq \frac{\sqrt{f(t_0)} e^{\frac{q(t_0)}2}}{1-\frac\kappa2} \frac{t-t_0}{ \sqrt{f(t)}}.
\end{equation}
\end{lemma}
\proof Recalling (\ref{ineq-f}), we have
$$ \int_{t_0}^{t}f(r)^{-\frac12}  dr
 \leq f(t)^{-\frac12}t^\frac\kappa2\int_{t_0}^{t}r^{-\frac\kappa2} dr
 =\frac{tf(t)^{-\frac12}}{1-\frac\kappa2}
 \left(1-\left(\frac{t_0}{t}\right)^{1-\frac\kappa2}\right). $$
Since $\kappa\in ]0,2[$, we obtain
$$
 \int_{t_0}^{t}f(r)^{-\frac12} dr\leq \frac{tf(t)^{-\frac12}}{1-\frac\kappa2}
 \left(1-\frac{t_0}{t}\right)=\frac{t-t_0}{1-\frac\kappa2}f(t)^{-\frac12}.
 \label{eq:inversefupper}
$$
Since $f\in\CM([t_0,\infty[)$ is non-decreasing, we obtain
$$
 \int_{t_0}^{t}f(r)^{-\frac12} dr
 \geq(t-t_0)f(t)^{-\frac12}.
$$
Recalling (\ref{equi-Hf+}) for $f\in\CM([t_0,\infty[)$, we have (\ref{eq:I0}) for every $t\geq t_0$.
%\[h_-:=\frac{e^{-q(t_0)/2}}{\sqrt{f(t_0)}},
%\qquad h_+:=\frac{e^{q(t_0)/2}}{\sqrt{f(t_0)}}.
%\]
%Combining these inequalities in Lemma \ref{fulu1}, we obtain
\qed

%\smallskip

\begin{lemma}\label{lemmaa5}
For $f\in\CM([t_0,\infty[)$ satisfying (\ref{condisupp1}), there exists $T_1>t_0$ such that
the integral $I(t)$ defined in (\ref{eq: I}) satisfies
 \begin{equation} \label{app_ineq1}
 I(t) \simeq (t-t_0)f(t)^{-\frac12},\qquad \forall \ t\geq T_1. 
\end{equation}
\end{lemma}
\proof Let $ g=H^{-1}$. Recalling the definition of $H(t)=H_f(t)$ in (\ref{Hft}), we have $g(t_0)=1$ and $g'\leq0$ since 
$H'_f(t)=  q(t)\cos^2(t)H_f(t)\geq 0$.
Integration by parts gives
$$\int_{t_0}^{t}g(r)\cos(2r) \, dr
 =\frac12[g(r)\sin(2r)]_{t_0}^{t}
 -\frac12\int_{t_0}^{t}g'(r)\sin(2r) \, dr.$$
Consequently,
\begin{eqnarray*}
 \left|\int_{t_0}^{t}g(r)\cos(2r) \, dr\right|
 &\leq&\frac12g(t)+\frac12g(t_0)
       +\frac12\int_{t_0}^{t}|g'(r)| dr\nonumber\\
 &=&\frac12g(t)+\frac12+\frac12(1-g(t)) \ = \ 1.
 \label{eq:cosinverse}
\end{eqnarray*}
Combining this bound with the identity $\sin^2(r)=(1-\cos(2r))/2$,
$$
 \frac12\CI_0(t)-\frac12\leq I(t)\leq \CI_0(t),\qquad \CI_0(t):=\int_{t_0}^{t}\frac{ dr}{H(r)}.
$$
By (\ref{Hf-sublinear}) and (\ref{eq:I0}),
$$\CI_0(t)\geq \sqrt{f(t_0)} e^{-\frac{q(t_0)}2}\frac{t-t_0}{\sqrt{f(t)}} 
 \gtrsim (t-t_0)t^{-\frac\kappa2}\to\infty,\quad {\rm as} \ t\to\infty.$$
Choose $T_1>t_0$  such that $\CI_0(t)\geq2$  for $t\geq T_1$.
Then
$$
 \frac14\CI_0(t)\leq I(t)\leq \CI_0(t),\qquad \forall \  t\geq T_1.
$$
Combining this estimate with Lemma \ref{lemmaA3}, we complete the proof. 
\qed

\begin{lemma}\label{lemma_app_inez}
For $f\in\CM([t_0,\infty[)$ satisfying (\ref{condisupp1}), the integrals $B(t)$ and $D(t)$ defined in (\ref{eq:B}) and (\ref{eq:D}) satisfy
$$  \left|B(t)\right|\leq H(t),\quad \left|D(t)\right|\lesssim 1, \qquad \forall \ t\geq t_0.$$
 \end{lemma}

\proof
Integration by parts on $ B(t)$ gives
$$
 B(t)=-\frac12[\cos(2r) H(r)]_{t_0}^{t} +\frac12\int_{t_0}^{t}\cos(2r)H'(r)dr.$$
Using $H'\geq0$ and $H(t_0)=1$, we obtain
\begin{eqnarray*}
 |B(t)|
 &\leq&\frac12(H(t)+1)+\frac12\int_{t_0}^{t}H'(r)dr\\
 &=&\frac12(H(t)+1)+\frac12(H(t)-1) \ = \ H(t).
\end{eqnarray*}

For every finite $t$, the integrand of $D(t)$ is continuous on the compact
triangle $t_0\leq r\leq\sigma\leq t$. By Fubini's theorem,
$$ D(t)=\int_{t_0}^{t}\frac{\sin(2r) q(r)}{H(r)^2}
 B(r)dr. $$
Since $q\geq0$ and $|B|\leq H$, we have
$$ |D(t)|\leq  \int_{t_0}^{t}\frac{q(r)}{H(r)}dr\lesssim  \int_{t_0}^{t}\frac{f'(r)}{f(r)^{\frac32}}
 dr\leq \frac2{\sqrt{f(t_0)}},\quad \forall \ t\geq t_0.\qed$$

\end{document}